\documentclass[a4paper,12pt]{article}

\usepackage{amsmath,amsfonts,amssymb,amsthm,anysize,mathrsfs}
\usepackage{enumerate}
\usepackage{epsfig}
\usepackage{supertabular}
\usepackage{graphicx}
\usepackage{color}
\usepackage{xcolor}
\usepackage{booktabs}
\usepackage{array}
\usepackage{makecell}
\usepackage{cases}
\usepackage{indentfirst}
\usepackage[a4paper,total={160mm,240mm}]{geometry}
\usepackage[none]{hyphenat}
\usepackage{url,hyperref}
\usepackage{tablefootnote}
\usepackage{setspace}
\usepackage{authblk}

\hypersetup{
colorlinks=true,
linkcolor=blue,
citecolor=blue,
urlcolor=black,
pdftitle={The generalized Griesmer and antiGriesmer bounds},
pdfauthor={Haihua Deng}
}

\newtheorem{theorem}{Theorem}[section]
\newtheorem{corollary}[theorem]{Corollary}
\newtheorem{lemma}[theorem]{Lemma}
\newtheorem{proposition}[theorem]{Proposition}

\theoremstyle{definition}
\newtheorem{remark}[theorem]{Remark}

\providecommand{\keywords}[1]{\textbf{Keywords.} #1}
\providecommand{\MSC}[1]{\textbf{MSC 2020.} #1}

\begin{document}

\title{The generalized Griesmer and antiGriesmer bounds}
\author{Haihua~Deng\thanks{School of Mathematical Sciences, Zhejiang University, 866 Yuhangtang Road, Hangzhou 310058, Zhejiang, China. \texttt{haihua.deng@zju.edu.cn}.}}
\date{}
\maketitle

\begin{abstract}
We present three proofs of the generalized Griesmer bound together with the corresponding proofs of the generalized antiGriesmer bound. The first proof follows from inequalities relating consecutive minimum and maximum subcode support weights. We also write the projective construction of Tsfasman and Vl\u{a}du\c{t} as a residual code argument and express the geometric proof of Kurz, Landjev, and Rousseva in terms of shortened subcodes. In addition, complements in repeated simplex codes show that the two bounds are equivalent. The residual and shortening arguments also determine the consequences of equality for the resulting residual codes and shortened subcodes. Finally, the complement relation transfers known divisibility results for Griesmer codes to antiGriesmer codes.
\end{abstract}

\noindent\keywords{generalized Hamming weight; generalized Griesmer bound; generalized antiGriesmer bound; residual code; shortened subcode; repeated simplex codes}

\medskip
\noindent\MSC{94B05; 94B65; 51E20}

\section{Introduction}\label{sec:introduction}

Throughout this paper, $q$ is a prime power and $\mathbb{F}_q$ is the finite field with $q$ elements. An $[n,k,d]_q$ code is a $k$-dimensional subspace of $\mathbb{F}_q^n$ with minimum Hamming distance $d$. If the minimum distance is unspecified, we simply write an $[n,k]_q$ code.

The Griesmer bound is one of the basic lower bounds on the length of a linear code: If $C$ is an $[n,k,d]_q$ code, then $n\geq\sum_{i=0}^{k-1}\left\lceil\frac{d}{q^i}\right\rceil$. Griesmer \cite[Theorem~5]{Griesmer1960} proved the binary case, and Solomon and Stiffler \cite[Theorem~1']{SolomonStiffler1965} proved the bound for arbitrary finite fields. A code attaining equality in the Griesmer bound is called a \textit{Griesmer code}.

Wei \cite{Wei1991} introduced generalized Hamming weights to measure the supports of higher-dimensional subcodes. If $d_r(C)$ denotes the minimum support size of an $r$-dimensional subcode of $C$, then the generalized Griesmer bound states that
\[
n\geq d_r(C)+\sum_{i=1}^{k-r}\left\lceil\frac{d_r(C)}{q^i{r\brack 1}_q}\right\rceil,
\]
where ${r\brack 1}_q=(q^r-1)/(q-1)$. This bound follows from the work of Helleseth, Kl{\o}ve, and Ytrehus \cite[Theorem~5]{HellesethKloveYtrehus1992} and Helleseth, Kl{\o}ve, Levenshtein, and Ytrehus \cite[Theorem~1]{HellesethKloveLevenshteinYtrehus1995}. The latter proof compares consecutive generalized Hamming weights by counting subcodes. Tsfasman and Vl\u{a}du\c{t} \cite[Corollary~3.5]{TsfasmanVladut1995} gave a projective proof based on their first construction in \cite[Proposition~3.2]{TsfasmanVladut1995}, which is the residual code obtained from a subcode of minimum support. Kurz, Landjev, and Rousseva \cite[Theorem~1]{KurzLandjevRousseva2026} gave a geometric proof by projection from a point. In code terms, this is the shortened subcode associated with the chosen point.

The same projective multiset also determines the maximum support weights. Tsfasman and Vl\u{a}du\c{t} \cite[Theorem~2.1]{TsfasmanVladut1995} and Dodunekov and Simonis \cite{DodunekovSimonis1998} describe the correspondence between full-length linear codes and spanning multisets of points in $\mathrm{PG}(k-1,q)$. Under this correspondence, $d_r(C)$ is determined by the largest intersection with a subspace of codimension $r$, whereas the maximum $r$-dimensional subcode support weight $\delta_r(C)$ is determined by the smallest such intersection. Thus, both parameters are determined by the intersection numbers of the same projective multiset.

For $r=1$, the parameter $\delta_1(C)$ is the maximum Hamming weight of $C$ and, since $C$ is linear, also its diameter. Thus, the study of $\delta_1(C)$ is naturally connected with the classical diameter problem for anticodes, which need not be linear. Delsarte \cite{Delsarte1973} established the code-anticode bound. In the binary Hamming space, Kleitman \cite{Kleitman1966} determined the largest anticodes of prescribed diameter, and Ahlswede and Khachatrian \cite{AhlswedeKhachatrian1998} later proved the corresponding $q$-ary diametric theorem. In the linear setting, Nogin \cite{Nogin1999} studied higher weights of anticodes and related them to maximum subcode support weights. More recently, Chen and Xie \cite[Theorem~2.2]{ChenXie2025} proved an antiGriesmer bound for projective linear anticodes subject to a restriction on the length. This result was extended in two different directions. Xie, Chen, Ding, and Li \cite{XieChenDingLi2026} established the generalized antiGriesmer bound for projective linear codes in terms of maximum $r$-dimensional subcode support weights and studied the corresponding subcode support weight distributions. For $r=1$, Zhang, Chen, Lin, and Liu \cite[Theorem~1.1]{ZhangChenLinLiu2026} removed the restriction on the length and proved the ordinary antiGriesmer bound for every full-length linear anticode.

The relation between the two bounds is most transparent under complementation. Chen and Xie \cite{ChenXie2025} introduced simplex complementary codes in the projective case, and Xie, Chen, Ding, and Li \cite{XieChenDingLi2026} determined the subcode support weight distributions of these codes. For projective multisets, Kurz, Landjev, and Rousseva \cite[Section~2]{KurzLandjevRousseva2026} record the relation $\mathcal{K}'=s-\mathcal{K}$ and the resulting formulas for maximum and minimum intersection numbers. We use the same relation for complements in repeated simplex codes. It gives the corresponding relation between minimum and maximum subcode support weights and proves the equivalence of the generalized Griesmer and antiGriesmer bounds.

The purpose of this paper is to present these three proofs and the complement relation in a common notation and to examine the equality cases of the two bounds. We recall the known arguments for generalized Hamming weights and give their counterparts for maximum subcode support weights. The residual and shortening arguments have the additional advantage that, when equality holds, they determine the parameters of the resulting codes. The complement argument also yields a simple divisibility consequence. Ward introduced divisible codes in \cite{Ward1981} and proved a divisibility theorem for Griesmer codes in \cite[Theorem~1]{Ward1998}. Deng, Huang, and Xiang \cite[Theorems~1.13 and 1.14]{DengHuangXiang2026} obtained further divisibility results. Since the weights of corresponding codewords in a code and its complement sum to the constant weight of a repeated simplex code, these results carry over to antiGriesmer codes.

The paper is organized as follows. In Section~\ref{sec:preliminaries}, we fix the notation and describe residual codes and shortened subcodes in terms of projective multisets. In Section~\ref{sec:hierarchy}, we derive both bounds from inequalities between consecutive support weights. In Section~\ref{sec:complements}, we study complements in repeated simplex codes and obtain the divisibility consequence. In Section~\ref{sec:residuals}, we give the residual proofs and determine the equality cases. In Section~\ref{sec:projective-points}, we give the shortening proof and determine its equality cases.

\section{Preliminaries and notation}\label{sec:preliminaries}

In the sequel, the code $C$ we consider will be linear. For $c=(c_1,\ldots,c_n)\in C$, define $\mathrm{supp}(c)=\{j:c_j\ne0\}$ and $\mathrm{wt}(c)=|\mathrm{supp}(c)|$. The Hamming distance between $x,y\in\mathbb{F}_q^n$ is $\mathrm{wt}(x-y)$; hence, the minimum distance of $C$ is the minimum weight of a nonzero codeword. With respect to the dot product $x\cdot y=\sum_{i=1}^n x_i y_i$, the \textit{dual code} of $C$ is
\[
C^\perp=\{x\in\mathbb{F}_q^n:x\cdot c=0\text{ for every }c\in C\}.
\]
For a subcode $A\leq C$, put $\mathrm{supp}(A)=\bigcup_{a\in A}\mathrm{supp}(a)$, $w(A)=|\mathrm{supp}(A)|$. The \textit{effective length} of $C$ is $n(C)=w(C)$. Thus, $C$ has \textit{full length} if $n(C)=n$, or equivalently, if no coordinate is identically zero on $C$. The \textit{diameter} of $C$ is the largest Hamming distance between two codewords. Since $C$ is linear, its diameter is equal to its maximum Hamming weight. If this value is $\delta$, we also call $C$ an $[n,k,\delta]_q$ \textit{anticode}. A positive integer $\Delta$ is a \textit{divisor of $C$} if $\Delta$ divides $\mathrm{wt}(c)$ for every $c\in C$; in this case, $C$ is \textit{$\Delta$-divisible}.

For $1\leq r\leq k$, define the \textit{$r$-th generalized Hamming weight}, also called the \textit{$r$-th minimum support weight}, and the \textit{maximum $r$-dimensional subcode support weight} by
\[
d_r(C)=\min_{\substack{A\leq C\\ \dim A=r}}w(A),\qquad \delta_r(C)=\max_{\substack{A\leq C\\ \dim A=r}}w(A),
\]
respectively. We use the term \textit{maximum $r$-dimensional subcode support weight} following \cite{XieChenDingLi2026}. We write $d_r$ and $\delta_r$ when the code is clear. Thus, $d_1$ is the minimum distance, while $\delta_1$ is the maximum Hamming weight, or equivalently, the diameter of $C$. Wei \cite[Theorem~1]{Wei1991} proved that $1\leq d_1<\cdots<d_k=n(C)$. The maximum support weights are nondecreasing and satisfy $\delta_k=n(C)$. Hence, we have $d_k(C)=\delta_k(C)=n$ when $C$ is of full length.

For integers $0\leq v\leq u$, the \textit{Gaussian coefficient} ${u\brack v}_q$ is the number of $v$-dimensional subspaces of $\mathbb{F}_q^u$. In particular, ${r\brack 1}_q=(q^r-1)/(q-1)$. For $x\in\mathbb{Z}_{\ge 0}$ and $1\leq r\leq k$, define
\begin{equation}\label{eq:griesmer-functions}
\begin{aligned}
g_q^{(r)}(k,x)&=x+\sum_{i=1}^{k-r}\left\lceil\frac{x}{q^i{r\brack 1}_q}\right\rceil, &
 a_q^{(r)}(k,x)&=x+\sum_{i=1}^{k-r}\left\lfloor\frac{x}{q^i{r\brack 1}_q}\right\rfloor.
\end{aligned}
\end{equation}
An empty sum is zero. When $r=1$, write $g_q(k,x)=g_q^{(1)}(k,x)$ and $a_q(k,x)=a_q^{(1)}(k,x)$. A full-length code $C$ is called an \textit{$r$-th generalized Griesmer code} if $n=g_q^{(r)}(k,d_r)$, and an \textit{$r$-th generalized antiGriesmer code} if $n=a_q^{(r)}(k,\delta_r)$. When $r=1$, we use the terms Griesmer code and \textit{antiGriesmer code}, respectively.

We next recall the projective description of a full-length code. Let $\mathcal{P}$ be the point set of $\mathrm{PG}(k-1,q)$. A \textit{multiset of points} is a map $\mathcal{M}:\mathcal{P}\to\mathbb{Z}_{\geq0}$. For $S\subseteq\mathcal{P}$, put $\mathcal{M}(S)=\sum_{P\in S}\mathcal{M}(P)$, and write
\[
|\mathcal{M}|=\mathcal{M}(\mathcal{P}),\qquad \gamma(\mathcal{M})=\max_{P\in\mathcal{P}}\mathcal{M}(P),\qquad \mu(\mathcal{M})=\min_{P\in\mathcal{P}}\mathcal{M}(P).
\]
The multiset is \textit{spanning} if its points of positive multiplicity span $\mathrm{PG}(k-1,q)$. If $G$ is a generator matrix of a full-length code $C$, its nonzero columns define a spanning multiset $\mathcal{M}_C$, where proportional columns are counted with multiplicity. Conversely, every spanning multiset gives a full-length code, uniquely up to a permutation of the coordinates and multiplication of individual coordinates by nonzero field elements; see \cite{DodunekovSimonis1998} for this correspondence between linear codes and multisets of points in projective spaces. The code $C$ is \textit{projective} if and only if $\gamma(\mathcal{M}_C)=1$. The $q$-ary \textit{simplex code} of dimension $k$ is generated by a matrix whose columns contain one nonzero representative of each point of $\mathrm{PG}(k-1,q)$. It has length ${k\brack 1}_q$, and every nonzero codeword has weight $q^{k-1}$.

Let $U\leq\mathbb{F}_q^k$. If $\dim U=r$, then $\{xG:x\in U\}$ is an $r$-dimensional subcode of $C$, and
\begin{equation}\label{eq:geometric-support}
w(\{xG:x\in U\})=\mathcal{M}_C\bigl(\mathcal{P}\setminus\mathrm{PG}(U^\perp)\bigr).
\end{equation}
The set $\mathcal{P}\setminus\mathrm{PG}(U^\perp)$ has ${k\brack 1}_q-{k-r\brack 1}_q=q^{k-r}{r\brack 1}_q$ points. For an integer $j$, put
\[B_j^{(r)}(C)=\#\{A\leq C:\dim A=r,\ w(A)=j\}.\]
Following \cite{XieChenDingLi2026}, the sequence $\left(B_j^{(r)}(C)\right)_j$ is the \textit{$r$-dimensional subcode support weight distribution}. Its smallest and largest nonzero indices are $d_r(C)$ and $\delta_r(C)$.

Let $\lambda\geq\gamma(\mathcal{M})$, and let $\lambda\mathcal{P}$ denote the constant multiset of multiplicity $\lambda$. This is the projective multiset of the $\lambda$-fold repetition of the $q$-ary simplex code of dimension $k$. The \textit{complement of $\mathcal{M}$ in $\lambda\mathcal{P}$} is $\mathcal{M}^{[\lambda]}=\lambda\mathcal{P}-\mathcal{M}$. Thus, $\mathcal{M}^{[\lambda]}(P)=\lambda-\mathcal{M}(P)$ for every $P\in\mathcal{P}$. If $\mathcal{M}^{[\lambda]}$ is spanning, denote its associated code by $C^{[\lambda]}$. This condition is automatic when $\lambda>\gamma(\mathcal{M})$.

We shall use two ways to reduce the dimension. For a subcode $A\leq C$, define the \textit{residual code of $C$ with respect to $A$} by
\[
\mathrm{Res}(C,A)=\{(c_j)_{j\notin\mathrm{supp}(A)}:c\in C\}.
\]
This code is obtained by puncturing the coordinates in $\mathrm{supp}(A)$. Denote the puncturing map by $\pi_A:C\to\mathrm{Res}(C,A)$. When $A=\langle c\rangle$, this is the usual \textit{residual code of $C$ with respect to $c$}, denoted by $\mathrm{Res}(C,c)$. Coordinates that are identically zero after puncturing are deleted, so residual codes are always written with their effective lengths.

Let $G_1,\ldots,G_n$ be the columns of $G$. For a point $P=\langle p\rangle$ of $\mathrm{PG}(k-1,q)$, define the \textit{shortened subcode of $C$ associated with $P$} by
\[
C_P=\{(xG_j)_{\langle G_j\rangle\ne P}:x\in p^\perp\}.
\]
This terminology includes the case $\mathcal{M}_C(P)=0$. The map from $p^\perp$ to $C_P$ is injective: if all retained coordinates of $xG$ vanish, then the coordinates represented by $P$ also vanish because $x\in p^\perp$, so $xG=0$ and hence $x=0$. No retained coordinate is identically zero on $C_P$, since this would force its column to lie in $(p^\perp)^\perp=\langle p\rangle$. Thus, $C_P$ is a full-length $[n-\mathcal{M}_C(P),k-1]_q$ code. If $\mathcal{M}_C(P)>0$, it is obtained by ordinary shortening at any coordinate represented by $P$ and deleting the other coordinates represented by $P$, which become identically zero. If $\mathcal{M}_C(P)=0$, no coordinate is deleted and $C_P=\{xG:x\in p^\perp\}$ is a codimension-one subcode of $C$.

Both constructions have standard projective interpretations. If $A=\{xG:x\in U\}$, then the coordinates outside $\mathrm{supp}(A)$ are precisely the columns lying in $\mathrm{PG}(U^\perp)$. Thus, $\mathrm{Res}(C,A)$ is represented by the restriction of $\mathcal{M}_C$ to this subspace. This is the first construction of Tsfasman and Vl\u{a}du\c{t} \cite[Proposition~3.2]{TsfasmanVladut1995}. For $C_P$, choose coordinates so that $P=\langle e_k\rangle$ and project from $P$ onto the hyperplane $x_k=0$. Deleting the last entry of every column not representing $P$ gives a generator matrix of $C_P$, while the columns representing $P$ disappear. Kurz, Landjev, and Rousseva use this projection in the proof of \cite[Theorem~1]{KurzLandjevRousseva2026}.

Finally, we list the elementary identities used below.

\begin{lemma}\label{lem:rounding}
Let $a,b,D,m$ be positive integers, and let $x$ be a nonnegative integer. Then
\begin{align}
\left\lceil\frac{\left\lceil x/a\right\rceil}{b}\right\rceil&=\left\lceil\frac{x}{ab}\right\rceil,&
\left\lfloor\frac{\left\lfloor x/a\right\rfloor}{b}\right\rfloor&=\left\lfloor\frac{x}{ab}\right\rfloor,\label{eq:nested-rounding}\\
\left\lceil\frac{x+\left\lceil x/D\right\rceil}{m(D+1)}\right\rceil&=\left\lceil\frac{x}{mD}\right\rceil,&
\left\lfloor\frac{x+\left\lfloor x/D\right\rfloor}{m(D+1)}\right\rfloor&=\left\lfloor\frac{x}{mD}\right\rfloor.\label{eq:shifted-rounding}
\end{align}
\end{lemma}

\begin{proof}
The identities in \eqref{eq:nested-rounding} follow from the definitions of the ceiling and floor functions. For the first identity in \eqref{eq:shifted-rounding}, put $y=\left\lceil x/D\right\rceil$. Then $(y-1)(D+1)<x+y\leq y(D+1)$, so $\left\lceil(x+y)/(D+1)\right\rceil=y$. The result follows by applying the first identity in \eqref{eq:nested-rounding} with divisor $m$. The floor identity is proved in the same way with $y=\left\lfloor x/D\right\rfloor$.
\end{proof}

\section{The averaging proof}\label{sec:hierarchy}

We begin with the counting argument of Helleseth, Kl{\o}ve, Levenshtein, and Ytrehus \cite[Theorem~1]{HellesethKloveLevenshteinYtrehus1995}. Their inequality concerns minimum support weights. The same count gives the corresponding inequality for maximum support weights.

\begin{proposition}\label{prop:support-comparison}
Let $C$ be an $[n,k]_q$ code. If $1\le s\le r\le k$, then
\[
(q^r-q^{r-s})d_r(C)\ge(q^r-1)d_s(C),\qquad (q^r-q^{r-s})\delta_r(C)\le(q^r-1)\delta_s(C).
\]
\end{proposition}

\begin{proof}
Fix an $r$-dimensional subcode $A$. Each coordinate in $\mathrm{supp}(A)$ defines a nonzero linear functional on $A$, whose kernel is a hyperplane. Among the ${r\brack s}_q$ subspaces of dimension $s$, exactly ${r-1\brack s}_q$ are contained in this kernel. Thus, the proportion of $s$-dimensional subspaces on which the functional is nonzero is $(q^r-q^{r-s})/(q^r-1)$. It follows that the average support weight of the $s$-dimensional subcodes of $A$ is $(q^r-q^{r-s})w(A)/(q^r-1)$. Taking $w(A)=d_r(C)$ gives the first inequality, while taking $w(A)=\delta_r(C)$ gives the second.
\end{proof}

The following elementary calculation allows the adjacent inequalities to be iterated.

\begin{lemma}\label{lem:adjacent-arithmetic}
For every nonnegative integer $x$ and every integer $2\le r\le k$, we have
\[
g_q^{(r)}\left(k,\left\lceil\frac{(q^r-1)x}{q^r-q}\right\rceil\right)=g_q^{(r-1)}(k,x),\qquad a_q^{(r)}\left(k,\left\lfloor\frac{(q^r-1)x}{q^r-q}\right\rfloor\right)=a_q^{(r-1)}(k,x).
\]
\end{lemma}

\begin{proof}
Since ${r\brack 1}_q=q{r-1\brack 1}_q+1$, we have
\[
\left\lceil\frac{(q^r-1)x}{q^r-q}\right\rceil=x+\left\lceil\frac{x}{q{r-1\brack 1}_q}\right\rceil,\qquad \left\lfloor\frac{(q^r-1)x}{q^r-q}\right\rfloor=x+\left\lfloor\frac{x}{q{r-1\brack 1}_q}\right\rfloor.
\]
Applying Lemma~\ref{lem:rounding} with $D=q{r-1\brack 1}_q$ shows that for every $i\ge1$,
\[
\left\lceil\frac{x+\left\lceil x/(q{r-1\brack 1}_q)\right\rceil}{q^i{r\brack 1}_q}\right\rceil=\left\lceil\frac{x}{q^{i+1}{r-1\brack 1}_q}\right\rceil,\qquad \left\lfloor\frac{x+\left\lfloor x/(q{r-1\brack 1}_q)\right\rfloor}{q^i{r\brack 1}_q}\right\rfloor=\left\lfloor\frac{x}{q^{i+1}{r-1\brack 1}_q}\right\rfloor.
\]
Substituting these identities into \eqref{eq:griesmer-functions} gives the two formulas.
\end{proof}

Taking $s=r-1$ in Proposition~\ref{prop:support-comparison} and iterating gives the two bounds at once.

\begin{theorem}\label{thm:hierarchy}
Let $C$ be an $[n,k]_q$ code of full length. Then the following inequalities hold:
\begin{align*}
n=g_q^{(k)}(k,d_k)&\ge g_q^{(k-1)}(k,d_{k-1})
\ge\cdots\ge g_q^{(1)}(k,d_1),\\
n=a_q^{(k)}(k,\delta_k)&\le a_q^{(k-1)}(k,\delta_{k-1})
\le\cdots\le a_q^{(1)}(k,\delta_1).
\end{align*}
Therefore, for every $1\le r\le k$,
\begin{align}
n&\ge d_r+\sum_{i=1}^{k-r}\left\lceil\frac{d_r}{q^i{r\brack 1}_q}\right\rceil,\label{eq:generalized-griesmer}\\
n&\le \delta_r+\sum_{i=1}^{k-r}\left\lfloor\frac{\delta_r}{q^i{r\brack 1}_q}\right\rfloor.\label{eq:generalized-antigriesmer}
\end{align}
\end{theorem}

\begin{proof}
Substituting $s=r-1$ into Proposition~\ref{prop:support-comparison} yields
\[
d_r\ge\left\lceil\frac{(q^r-1)d_{r-1}}{q^r-q}\right\rceil,\qquad \delta_r\le\left\lfloor\frac{(q^r-1)\delta_{r-1}}{q^r-q}\right\rfloor.
\]
Since both functions in \eqref{eq:griesmer-functions} are strictly increasing in their last variable, Lemma~\ref{lem:adjacent-arithmetic} gives
\[
g_q^{(r)}(k,d_r)\ge g_q^{(r-1)}(k,d_{r-1}),\qquad a_q^{(r)}(k,\delta_r)\le a_q^{(r-1)}(k,\delta_{r-1}).
\]
The two chains start at $n$, since $d_k=\delta_k=n$.
\end{proof}

The two chains also show that equality at one level propagates to every higher level.

\begin{corollary}\label{cor:equality-propagation}
If $C$ is an $r$-th generalized Griesmer code, then it is an $h$-th generalized Griesmer code for every $r\le h\le k$. Similarly, if $C$ is an $r$-th generalized antiGriesmer code, then it is an $h$-th generalized antiGriesmer code for every $r\le h\le k$.
\end{corollary}

\begin{proof}
For the Griesmer bound, each term $g_q^{(h)}(k,d_h)$ with $h\ge r$ lies between $g_q^{(r)}(k,d_r)=n$ and $g_q^{(k)}(k,d_k)=n$. Hence, every such term is equal to $n$. The same argument applies to the antiGriesmer chain.
\end{proof}

When equality holds, the adjacent inequalities are exact. We shall use the resulting formulas in Section~\ref{sec:residuals}.

\begin{corollary}\label{cor:exact-recursion}
If $C$ is an $r$-th generalized Griesmer code, then, for $r<h\le k$,
\begin{equation}\label{eq:griesmer-recursion}
d_h=\left\lceil\frac{(q^h-1)d_{h-1}}{q^h-q}\right\rceil.
\end{equation}
If $C$ is an $r$-th generalized antiGriesmer code, then, for $r<h\le k$,
\begin{equation}\label{eq:antigriesmer-recursion}
\delta_h=\left\lfloor\frac{(q^h-1)\delta_{h-1}}{q^h-q}\right\rfloor.
\end{equation}
\end{corollary}

\begin{proof}
Assume that $C$ is an $r$-th generalized Griesmer code. Corollary~\ref{cor:equality-propagation} and Lemma~\ref{lem:adjacent-arithmetic} give
\[
g_q^{(h)}(k,d_h)
=g_q^{(h)}\left(k,\left\lceil\frac{(q^h-1)d_{h-1}}{q^h-q}\right\rceil\right)
=n.
\]
Proposition~\ref{prop:support-comparison} gives the reverse inequality between the two arguments. Since $g_q^{(h)}(k,x)$ is strictly increasing in $x$, the arguments are equal, and \eqref{eq:griesmer-recursion} follows. Replacing $d_i$, $g_q^{(h)}$, and ceilings by $\delta_i$, $a_q^{(h)}$, and floors gives \eqref{eq:antigriesmer-recursion}.
\end{proof}

\section{The complement argument}\label{sec:complements}

We next consider complements in repeated simplex codes. Chen and Xie \cite{ChenXie2025} studied simplex complementary codes of projective linear codes, and Xie, Chen, Ding, and Li \cite{XieChenDingLi2026} determined their $r$-dimensional subcode support weight distributions. Kurz, Landjev, and Rousseva \cite[Section~2]{KurzLandjevRousseva2026} stated the corresponding relation for projective multisets in the form $\mathcal{K}'=s-\mathcal{K}$. We give the formulas below and then deduce the equivalence of the two bounds.

\begin{theorem}\label{thm:support-distribution-reflection}
Let $C$ be an $[n,k]_q$ code of full length. Let $\lambda\ge\gamma(\mathcal{M}_C)$, and suppose that $\mathcal{M}_C^{[\lambda]}$ is spanning. Then $n(C^{[\lambda]})=\lambda{k\brack 1}_q-n$. Moreover, for every $1\leq r\leq k$ and every $j$,
\[
B_j^{(r)}(C^{[\lambda]})=B_{\lambda q^{k-r}{r\brack 1}_q-j}^{(r)}(C).
\]
In particular,
\[
d_r(C^{[\lambda]})=\lambda q^{k-r}{r\brack 1}_q-\delta_r(C),\qquad \delta_r(C^{[\lambda]})=\lambda q^{k-r}{r\brack 1}_q-d_r(C).
\]
\end{theorem}

\begin{proof}
Summing the multiplicities over the ${k\brack 1}_q$ points of $\mathrm{PG}(k-1,q)$ gives the formula for $n(C^{[\lambda]})$. Let $U\le\mathbb{F}_q^k$ be an $r$-dimensional subspace. By \eqref{eq:geometric-support} and the point count above,
\begin{align*}
\mathcal{M}_C^{[\lambda]}\bigl(\mathcal{P}\setminus\mathrm{PG}(U^\perp)\bigr)
=\sum_{P\notin\mathrm{PG}(U^\perp)}\bigl(\lambda-\mathcal{M}_C(P)\bigr)=\lambda q^{k-r}{r\brack 1}_q
-\mathcal{M}_C\bigl(\mathcal{P}\setminus\mathrm{PG}(U^\perp)\bigr).
\end{align*}
Each $r$-dimensional subspace $U$ of $\mathbb{F}_q^k$ gives an $r$-dimensional subcode of both $C$ and $C^{[\lambda]}$, and the two support weights sum to $\lambda q^{k-r}{r\brack 1}_q$. This proves the identity for $B_j^{(r)}$. Taking the minimum and maximum gives the last two formulas.
\end{proof}

The next identity relates the two numerical bounds.

\begin{lemma}\label{lem:complement-arithmetic}
Let $x$ be a nonnegative integer. If $\lambda q^{k-r}{r\brack 1}_q-x\ge0$, then
\begin{equation*}
a_q^{(r)}\left(k,\lambda q^{k-r}{r\brack 1}_q-x\right)
=\lambda{k\brack 1}_q-g_q^{(r)}(k,x).
\end{equation*}
\end{lemma}

\begin{proof}
For any index $1\le i\le k-r$, the ratio
\[
\frac{\lambda q^{k-r}{r\brack 1}_q}{q^i{r\brack 1}_q}
=\lambda q^{k-r-i}
\]
is an integer. Hence,
\[
\left\lfloor
\frac{\lambda q^{k-r}{r\brack 1}_q-x}{q^i{r\brack 1}_q}
\right\rfloor
=\lambda q^{k-r-i}
-\left\lceil\frac{x}{q^i{r\brack 1}_q}\right\rceil.
\]
Using the identity
\[
{k\brack 1}_q
=q^{k-r}{r\brack 1}_q+\sum_{i=1}^{k-r}q^{k-r-i}
\]
and substituting into \eqref{eq:griesmer-functions} completes the proof.
\end{proof}

The complement formula and the preceding identity give the equivalence of the two bounds.

\begin{theorem}\label{thm:bound-equivalence}
Fix $q$, $k$ and $r$. The generalized Griesmer inequality for all full-length $[n,k]_q$ codes is equivalent to the generalized antiGriesmer inequality for all full-length $[n,k]_q$ anticodes.
\end{theorem}

\begin{proof}
Assume first that the generalized Griesmer inequality holds. Let $C$ be a full-length code and choose an integer $\lambda>\gamma(\mathcal{M}_C)$. Then the complement is spanning. By Theorem~\ref{thm:support-distribution-reflection}, we have
\[
n(C^{[\lambda]})=\lambda{k\brack 1}_q-n,\qquad d_r(C^{[\lambda]})=\lambda q^{k-r}{r\brack 1}_q-\delta_r(C).
\]
Applying the generalized Griesmer inequality to $C^{[\lambda]}$ and then using Lemma~\ref{lem:complement-arithmetic}, we obtain
\begin{align*}
\lambda{k\brack 1}_q-n
\ge g_q^{(r)}\left(k,\lambda q^{k-r}{r\brack 1}_q-\delta_r(C)\right)=\lambda{k\brack 1}_q-a_q^{(r)}(k,\delta_r(C)).
\end{align*}
This is the generalized antiGriesmer inequality for $C$.

Conversely, assume the generalized antiGriesmer inequality and apply it to the same complement. By Theorem~\ref{thm:support-distribution-reflection} and Lemma~\ref{lem:complement-arithmetic},
\begin{align*}
\lambda{k\brack 1}_q-n
\le a_q^{(r)}\left(k,\lambda q^{k-r}{r\brack 1}_q-d_r(C)\right)=\lambda{k\brack 1}_q-g_q^{(r)}(k,d_r(C)).
\end{align*}
This gives $n\geq g_q^{(r)}(k,d_r(C))$ and proves the converse implication.
\end{proof}

The same calculation also gives the equality cases.

\begin{corollary}\label{cor:equality-correspondence}
Under the hypotheses of Theorem~\ref{thm:support-distribution-reflection},
\begin{align*}
a_q^{(r)}\bigl(k,\delta_r(C^{[\lambda]})\bigr)-n(C^{[\lambda]})
&=n-g_q^{(r)}(k,d_r(C)),\\
n(C^{[\lambda]})-g_q^{(r)}\bigl(k,d_r(C^{[\lambda]})\bigr)
&=a_q^{(r)}(k,\delta_r(C))-n.
\end{align*}
Therefore, $C$ is an $r$-th generalized Griesmer code if and only if $C^{[\lambda]}$ is an $r$-th generalized antiGriesmer code. Similarly, $C$ is an $r$-th generalized antiGriesmer code if and only if $C^{[\lambda]}$ is an $r$-th generalized Griesmer code.
\end{corollary}

\begin{proof}
Both identities follow by substituting the formulas in Theorem~\ref{thm:support-distribution-reflection} and then applying Lemma~\ref{lem:complement-arithmetic}.
\end{proof}

\subsection{A divisibility consequence}\label{sec:divisibility}

We conclude this section with a consequence for divisible codes. If the constant weight of the repeated simplex code is divisible by $\Delta$, then complementation preserves $\Delta$-divisibility.

\begin{proposition}\label{prop:complement-divisibility}
Let $C$ be a full-length $[n,k]_q$ code, let $\lambda\geq\gamma(\mathcal{M}_C)$, and suppose that $\mathcal{M}_C^{[\lambda]}$ is spanning. If $\Delta\mid\lambda q^{k-1}$, then $C$ is $\Delta$-divisible if and only if its complement $C^{[\lambda]}$ is $\Delta$-divisible.
\end{proposition}

\begin{proof}
Choose generator matrices $G_C$ and $G_{C^{[\lambda]}}$ from the complementary projective multisets. For every nonzero $x\in\mathbb{F}_q^k$, the corresponding codewords satisfy
\[
\mathrm{wt}(xG_C)+\mathrm{wt}(xG_{C^{[\lambda]}})=\lambda q^{k-1}.
\]
Since $\Delta$ divides the right-hand side, the two weights are simultaneously divisible by $\Delta$.
\end{proof}

Combining this observation with \cite[Theorems~1.13 and 1.14]{DengHuangXiang2026} yields the following result.

\begin{corollary}\label{cor:antigriesmer-divisibility}
Let $q=p^f$, where $p$ is prime and $f\geq1$. Let $e$ be a nonnegative integer, and let $C$ be an antiGriesmer code with parameters $[a_q(k,\delta),k,\delta]_q$.
\begin{enumerate}
\item If $q^e\mid\delta$, then $C$ is $p^e$-divisible.
\item If $p^e\mid\delta$, then $C$ is $\left\lceil p^{e-(f-1)(q-2)}\right\rceil$-divisible.
\end{enumerate}
\end{corollary}

\begin{proof}
Assume first that $q^e$ divides $\delta$. Choose a positive integer $\lambda>\gamma(\mathcal{M}_C)$ divisible by $q^e$, and put $D=C^{[\lambda]}$. By Corollary~\ref{cor:equality-correspondence} with $r=1$, the code $D$ is a Griesmer code. Its minimum distance is $d=\lambda q^{k-1}-\delta$, which is divisible by $q^e$. By \cite[Theorem~1.13]{DengHuangXiang2026}, $D$ is $p^e$-divisible. Since $p^e$ divides $\lambda q^{k-1}$, Proposition~\ref{prop:complement-divisibility} shows that $C$ is also $p^e$-divisible.

Now assume that $p^e$ divides $\delta$. Choose $\lambda>\gamma(\mathcal{M}_C)$ divisible by $p^e$. Then $D=C^{[\lambda]}$ is a Griesmer code with minimum distance $d=\lambda q^{k-1}-\delta$, and $p^e$ divides $d$. Put $\Delta=\left\lceil p^{e-(f-1)(q-2)}\right\rceil$. By \cite[Theorem~1.14]{DengHuangXiang2026}, $D$ is $\Delta$-divisible. If $e-(f-1)(q-2)<0$, then $\Delta=1$. Otherwise, $\Delta=p^{e-(f-1)(q-2)}$ divides $p^e$, and hence divides $\lambda q^{k-1}$. Proposition~\ref{prop:complement-divisibility} proves the second statement.
\end{proof}

\begin{remark}
The preceding corollary uses divisibility results for ordinary
Griesmer codes, corresponding to the case $r=1$. For $r>1$, equality
in the generalized Griesmer bound is expressed in terms of the support
weights of $r$-dimensional subcodes, whereas divisibility is a condition
on the weights of individual codewords. It is therefore natural to ask
whether an $r$-th generalized Griesmer code must
satisfy any corresponding divisibility condition. The complement
construction leads to the analogous question for $r$-th generalized
antiGriesmer codes.
\end{remark}

\section{The residual argument and equality cases}\label{sec:residuals}

The residual proof below is given by Tsfasman and Vl\u{a}du\c{t} \cite[Proposition~3.2]{TsfasmanVladut1995}. They apply it to the generalized Griesmer bound in \cite[Corollary~3.5]{TsfasmanVladut1995}. We write the construction in terms of residual codes, determine the equality case, and then give the analogous argument for maximum support weights.

\subsection{Minimum support weights}

We first consider minimum support weights. Here the puncturing kernel is exactly the chosen subcode, so the generalized Hamming weights of the residual code can be compared directly with those of the original code.

\begin{lemma}\label{lem:residual-minimum}
Let $A\le C$ have dimension $h<k$ and support weight $w$. If $w<d_{h+1}(C)$, then $\ker(\pi_A)=A$ and $\dim\mathrm{Res}(C,A)=k-h$. Moreover, $d_i\bigl(\mathrm{Res}(C,A)\bigr)\geq d_{i+h}(C)-w$ for all $1\leq i\leq k-h$.
The same conclusions hold when $w=d_h(C)$.
\end{lemma}

\begin{proof}
The kernel of $\pi_A$ consists of the codewords supported in $\mathrm{supp}(A)$, and it contains $A$. If its dimension were at least $h+1$, then it would contain an $(h+1)$-dimensional subcode with support weight at most $w<d_{h+1}(C)$, a contradiction. Thus, $\ker(\pi_A)=A$.

Let $B'\leq\mathrm{Res}(C,A)$ be an $i$-dimensional subcode, and choose an $i$-dimensional subspace $B\leq\pi_A^{-1}(B')$ complementary to $A$. Then $\dim(A+B)=h+i$ and $w(A+B)\leq w+w(B')$. Hence, $d_{h+i}(C)\leq w+w(B')$. Taking the minimum over $B'$ proves the inequality. The last statement follows from $d_h(C)<d_{h+1}(C)$.
\end{proof}

For $r=1$, this is the usual residual step in the proof of the Griesmer bound. The same induction works for every $r$.

\begin{proposition}\label{prop:residual-proof-griesmer}
Every full-length $[n,k]_q$ code satisfies the generalized Griesmer inequality \eqref{eq:generalized-griesmer}.
\end{proposition}

\begin{proof}
We argue by induction on $r$. The case $r=1$ is the ordinary Griesmer bound. Let $r\ge2$ and assume the result for smaller values of $r$. If $k=r$, then $n=d_k$.

Choose $1\le h<r$ and an $h$-dimensional subcode $A$ with $w(A)=d_h$, and put $R=\mathrm{Res}(C,A)$. Lemma~\ref{lem:residual-minimum} gives $\dim R=k-h$ and $d_{r-h}(R)\geq d_r-d_h$.
The induction hypothesis applied to $R$ gives
\begin{align*}
n-d_h
\ge g_q^{(r-h)}\bigl(k-h,d_{r-h}(R)\bigr)\ge d_r-d_h+\sum_{i=1}^{k-r}
\left\lceil\frac{d_r-d_h}{q^i{r-h\brack 1}_q}\right\rceil.
\end{align*}
Proposition~\ref{prop:support-comparison}, applied with $s=h$, gives $d_h\leq(q^r-q^{r-h})d_r/(q^r-1)$. Hence,
\begin{equation}\label{eq:residual-minimum-ratio}
\frac{d_r-d_h}{{r-h\brack 1}_q}
\ge\frac{d_r}{{r\brack 1}_q}.
\end{equation}
Applying \eqref{eq:residual-minimum-ratio} to each ceiling term and adding $d_h$ completes the proof.
\end{proof}

If equality holds in the generalized Griesmer bound, the same inequalities determine the parameters of the residual code. For $r=1$, see \cite[Theorem~3.3(1)]{DengHuangXiang2026}.

\begin{theorem}\label{thm:griesmer-residual-structure}
Let $C$ be an $r$-th generalized Griesmer code.
\begin{enumerate}
\item For $1\le h<r$, let $A\le C$ be an $h$-dimensional subcode with support weight $d_h$. The residual code $R=\mathrm{Res}(C,A)$ has dimension $k-h$ and satisfies
\[
n(R)=g_q^{(r-h)}(k-h,d_r-d_h),\qquad d_{r-h}(R)=d_r-d_h.
\]
Thus, $R$ is an $(r-h)$-th generalized Griesmer code.
\item If $r<k$, let $A\le C$ be an $r$-dimensional subcode with support weight $d_r$. The resulting residual code $R=\mathrm{Res}(C,A)$ has parameters
\[
\bigl[g_q(k-r,d_{r+1}-d_r),\ k-r,\ d_{r+1}-d_r\bigr]_q
\]
and is a Griesmer code.
\end{enumerate}
\end{theorem}

\begin{proof}
Let $1\le h<r$. By Lemma~\ref{lem:residual-minimum}, the residual code $R$ has dimension $k-h$ and satisfies $d_{r-h}(R)\geq d_r-d_h$.
Applying the generalized Griesmer bound to $R$ and using \eqref{eq:residual-minimum-ratio} gives
\begin{align*}
n-d_h
\ge g_q^{(r-h)}\bigl(k-h,d_{r-h}(R)\bigr)\ge g_q^{(r-h)}(k-h,d_r-d_h)\ge g_q^{(r)}(k,d_r)-d_h
=n-d_h.
\end{align*}
The first and last terms are equal; hence, every inequality is an equality. Since $g_q^{(r-h)}(k-h,x)$ is strictly increasing in $x$, it follows that $d_{r-h}(R)=d_r-d_h$ and $n(R)=g_q^{(r-h)}(k-h,d_r-d_h)$.

Now let $h=r$. Lemma~\ref{lem:residual-minimum} gives $\dim R=k-r$ and $d_1(R)\geq d_{r+1}-d_r$. By \eqref{eq:griesmer-recursion}, $d_{r+1}-d_r=\left\lceil d_r/(q{r\brack 1}_q)\right\rceil$.
Put $D=d_{r+1}-d_r$. Lemma~\ref{lem:rounding} gives
\begin{align*}
n-d_r=\sum_{i=1}^{k-r}\left\lceil\frac{d_r}{q^i{r\brack 1}_q}\right\rceil=D+\sum_{j=1}^{k-r-1}\left\lceil\frac{D}{q^j}\right\rceil
=g_q(k-r,D).
\end{align*}
The standard Griesmer bound applied to $R$ now forces $d_1(R)=D$, since $g_q(k-r,x)$ is strictly increasing in $x$. Thus, $R$ has the stated parameters and is a Griesmer code.
\end{proof}

\subsection{Maximum support weights}

For maximum support weights, the puncturing kernel may be larger than the chosen subcode. Thus, the residual code may have dimension smaller than $k-h$.

\begin{lemma}\label{lem:residual-maximum}
Assume $A\leq C$ has dimension $h$ and support weight $w$. Define $R=\mathrm{Res}(C,A)$ and let $k'=\dim R$. Then $\delta_i(R)\leq\delta_{i+h}(C)-w$ for all $1\leq i\leq k'$.
\end{lemma}

\begin{proof}
Let $K=\ker(\pi_A)$, so $A\le K$. For an $i$-dimensional subcode $B'\le R$, choose an $i$-dimensional subspace $B\le\pi_A^{-1}(B')$ complementary to $K$. Then $\dim(A+B)=h+i$. Since every coordinate in $\mathrm{supp}(A)$ is nonzero on some codeword of $A$, we have $w(A+B)=w+w(B')$. Hence, $w+w(B')\leq\delta_{h+i}(C)$. Taking the maximum over $B'$ proves the result.
\end{proof}

For the ordinary antiGriesmer bound, we shall use the following residual estimate.

\begin{lemma}\label{lem:maximum-codeword-residual}
Let $c\in C$ be a codeword of weight $\delta_1(C)=\delta$, and define $R=\mathrm{Res}(C,c)$. If $R\ne0$, then $\delta_1(R)\leq\left\lfloor\delta/q\right\rfloor$.
\end{lemma}

\begin{proof}
Assume that $R\ne0$. Then $n(C)>\delta$. Let $b'\in R$ be nonzero and choose a preimage $b\in C$. For any scalar $\alpha\in\mathbb{F}_q$, let $N_\alpha=\#\{j\in\mathrm{supp}(c):b_j=\alpha c_j\}$.
Each vector $b-\alpha c$ must be nonzero because it restricts to $b'$ outside $\mathrm{supp}(c)$. Thus, $\mathrm{wt}(b')+\delta-N_\alpha=\mathrm{wt}(b-\alpha c)\leq\delta$, and hence $\mathrm{wt}(b')\leq N_\alpha$. Summing these quantities over all $\alpha\in\mathbb{F}_q$ gives $q\,\mathrm{wt}(b')\leq\sum_{\alpha\in\mathbb{F}_q}N_\alpha=\delta$.
Taking the maximum over all nonzero $b'\in R$ proves the lemma.
\end{proof}

We now combine the two residual estimates.

\begin{proposition}\label{prop:residual-proof-antigriesmer}
Every $[n,k]_q$ anticode of full length satisfies the generalized antiGriesmer inequality \eqref{eq:generalized-antigriesmer}.
\end{proposition}

\begin{proof}
We first prove the case $r=1$ by induction on $k$. If $k=1$, then $n=\delta_1$. Let $c$ be a codeword of weight $\delta_1$, put $R=\mathrm{Res}(C,c)$, and write $k'=\dim R\le k-1$. If $k'=0$, then $n=\delta_1$. If $k'>0$, Lemma~\ref{lem:maximum-codeword-residual}, the induction hypothesis, and Lemma~\ref{lem:rounding} give
\begin{align*}
n-\delta_1\le a_q\left(k',\left\lfloor\frac{\delta_1}{q}\right\rfloor\right)=\sum_{i=1}^{k'}\left\lfloor\frac{\delta_1}{q^i}\right\rfloor
\le\sum_{i=1}^{k-1}\left\lfloor\frac{\delta_1}{q^i}\right\rfloor.
\end{align*}
This recovers the ordinary antiGriesmer bound.

Now let $r\ge2$ and proceed by induction on $r$. Choose a codeword $c$ of weight $\delta_1$ and let $R=\mathrm{Res}(C,c)$ have dimension $k'$. If $k'=0$, then $n=\delta_1\le\delta_r$. If $0<k'<r-1$, then the full length of $R$ and Lemma~\ref{lem:residual-maximum} give $n-\delta_1=\delta_{k'}(R)\leq\delta_{k'+1}(C)-\delta_1\leq\delta_r-\delta_1$. Hence, $n\leq\delta_r$, which is stronger than the required inequality.

It remains to consider $k'\ge r-1$. The induction hypothesis and Lemma~\ref{lem:residual-maximum} give
\[
n-\delta_1\le a_q^{(r-1)}\bigl(k',\delta_{r-1}(R)\bigr)\le a_q^{(r-1)}(k',\delta_r-\delta_1).
\]
Taking $s=1$ in Proposition~\ref{prop:support-comparison} gives
\begin{equation*}
\frac{\delta_r-\delta_1}{{r-1\brack 1}_q}
\le\frac{\delta_r}{{r\brack 1}_q}.
\end{equation*}
Since $k'-(r-1)\le k-r$, we have
\begin{align*}
n-\delta_1\le\delta_r-\delta_1
+\sum_{i=1}^{k'-(r-1)}
\left\lfloor\frac{\delta_r-\delta_1}{q^i{r-1\brack 1}_q}\right\rfloor\le\delta_r-\delta_1
+\sum_{i=1}^{k-r}
\left\lfloor\frac{\delta_r}{q^i{r\brack 1}_q}\right\rfloor.
\end{align*}
Adding $\delta_1$ to both sides completes the proof.
\end{proof}

Since the puncturing kernel may be larger than the chosen subcode, the equality statement must use the actual dimension of the residual code.

\begin{theorem}\label{thm:antigriesmer-residual-structure}
Let $C$ be an $r$-th generalized antiGriesmer code.
\begin{enumerate}
\item Let $1\le h<r$, and let $A\le C$ be an $h$-dimensional subcode with support weight $\delta_h$. Put $R=\mathrm{Res}(C,A)$, $k'=\dim R$, and $m=\min\{k',r-h\}$.
If $R$ is trivial ($R=0$), then $n=\delta_h=\delta_r$. If $R\ne0$, then $n(R)=a_q^{(m)}(k',\delta_r-\delta_h)$ and $\delta_m(R)=\delta_r-\delta_h$.
Thus, $R$ is an $m$-th generalized antiGriesmer code.
\item Suppose that $r<k$. Let $A\le C$ be an $r$-dimensional subcode with support weight $\delta_r$, and suppose that $R=\mathrm{Res}(C,A)$ is nonzero. If $k'=\dim R$, then $R$ has length $a_q(k',\delta_{r+1}-\delta_r)$, dimension $k'$, and maximum Hamming weight $\delta_{r+1}-\delta_r$. Hence, $R$ is an antiGriesmer code.
\end{enumerate}
\end{theorem}

\begin{proof}
Let $1\le h<r$. If $k'=0$, then $n-\delta_h=0$. Since $\delta_h\le\delta_r\le n$, it follows that $n=\delta_h=\delta_r$.

Assume that $k'>0$. If $k'<r-h$, then $m=k'$. Since $R$ has full length, Lemma~\ref{lem:residual-maximum} gives $n-\delta_h=\delta_{k'}(R)\leq\delta_{k'+h}(C)-\delta_h\leq\delta_r-\delta_h$.
Since $n\ge\delta_r$, equality holds throughout the chain. Hence, $n=\delta_r$, and the stated equalities follow because $m=k'$.

It remains to consider $k'\ge r-h$, so $m=r-h$. Theorem~\ref{thm:hierarchy} and Lemma~\ref{lem:residual-maximum} give
\[
n-\delta_h\le a_q^{(r-h)}\bigl(k',\delta_{r-h}(R)\bigr)\le a_q^{(r-h)}(k',\delta_r-\delta_h).
\]
Taking $s=h$ in Proposition~\ref{prop:support-comparison} gives $(\delta_r-\delta_h)/{r-h\brack 1}_q\leq\delta_r/{r\brack 1}_q$. Since $k'-(r-h)\le k-r$,
\[
a_q^{(r-h)}(k',\delta_r-\delta_h)\le\delta_r-\delta_h+\sum_{i=1}^{k-r}\left\lfloor\frac{\delta_r}{q^i{r\brack 1}_q}\right\rfloor=n-\delta_h.
\]
Hence, every inequality is an equality, and the first statement follows.

For the second part, Corollary~\ref{cor:equality-propagation} gives equality at level $r+1$. Apply the first part with $r+1$ in place of $r$ and with $h=r$. Then $m=1$, and the stated parameters follow.
\end{proof}

Notice that $k'=k-h$ holds only when $\ker(\pi_A)=A$. In general the kernel can be larger, which is why the statement is formulated in terms of $k'$.

\section{The shortening argument and equality cases}\label{sec:projective-points}

Projection from a point is one of the spoiling operations considered by Tsfasman and Vl\u{a}du\c{t} \cite[Section~III-B.3]{TsfasmanVladut1995}. Kurz, Landjev, and Rousseva \cite[Theorem~1]{KurzLandjevRousseva2026} use this operation to prove the generalized Griesmer bound. We write their argument in terms of the shortened subcode $C_P$ defined in Section~\ref{sec:preliminaries}; choosing a point of minimum multiplicity gives the generalized antiGriesmer bound.

\begin{lemma}\label{lem:point-averaging}
Let $A=\{xG:x\in U\}\le C$ be an $r$-dimensional subcode with support weight $w(A)$. Among the $q^{k-r}{r\brack 1}_q$ points of $\mathcal{P}\setminus\mathrm{PG}(U^\perp)$, there is a point whose multiplicity is at least $\left\lceil w(A)/(q^{k-r}{r\brack 1}_q)\right\rceil$ and a point whose multiplicity is at most $\left\lfloor w(A)/(q^{k-r}{r\brack 1}_q)\right\rfloor$.
Thus,
\[
\gamma(\mathcal{M}_C)\ge\left\lceil\frac{d_r(C)}{q^{k-r}{r\brack 1}_q}\right\rceil,\qquad \mu(\mathcal{M}_C)\le\left\lfloor\frac{\delta_r(C)}{q^{k-r}{r\brack 1}_q}\right\rfloor.
\]
\end{lemma}

\begin{proof}
By \eqref{eq:geometric-support}, the sum of the multiplicities over the $q^{k-r}{r\brack 1}_q$ points of $\mathcal{P}\setminus\mathrm{PG}(U^\perp)$ is $w(A)$. The two statements follow by averaging. Apply the first to a subcode of support weight $d_r(C)$ and the second to a subcode of support weight $\delta_r(C)$.
\end{proof}

A point of maximum multiplicity gives the generalized Griesmer bound, while a point of minimum multiplicity gives the generalized antiGriesmer bound.

\begin{proposition}\label{prop:projective-proof-griesmer}
Every $[n,k]_q$ code of full length satisfies the generalized Griesmer inequality \eqref{eq:generalized-griesmer}.
\end{proposition}

\begin{proof}
Fix $r$ and argue by induction on $k\ge r$. If $k=r$, then $n=d_k$. Suppose $k>r$, and let $P$ be a point of maximum multiplicity. Let $\gamma=\mathcal{M}_C(P)$. The resulting code $C_P$ has dimension $k-1$, an effective length of $n-\gamma$, and satisfies $d_r(C_P)\geq d_r(C)$.
The induction hypothesis and Lemma~\ref{lem:point-averaging} give
\begin{align*}
n-\gamma&\ge g_q^{(r)}(k-1,d_r(C_P))\ge g_q^{(r)}(k-1,d_r),\\
\gamma&\ge\left\lceil\frac{d_r}{q^{k-r}{r\brack 1}_q}\right\rceil.
\end{align*}
Adding these two inequalities gives \eqref{eq:generalized-griesmer}.
\end{proof}

\begin{proposition}\label{prop:projective-proof-antigriesmer}
Every $[n,k]_q$ anticode of full length satisfies the generalized antiGriesmer inequality \eqref{eq:generalized-antigriesmer}.
\end{proposition}

\begin{proof}
Fix $r$ and argue by induction on $k\ge r$. If $k=r$, then $n=\delta_k$. Suppose that $k>r$, and let $P$ be a point of minimum multiplicity and put $\mu=\mathcal{M}_C(P)$. The code $C_P$ has dimension $k-1$, length $n-\mu$, and satisfies $\delta_r(C_P)\leq\delta_r(C)$.
The induction hypothesis and Lemma~\ref{lem:point-averaging} give
\begin{align*}
n-\mu&\le a_q^{(r)}(k-1,\delta_r(C_P))\le a_q^{(r)}(k-1,\delta_r),\\
\mu&\le\left\lfloor\frac{\delta_r}{q^{k-r}{r\brack 1}_q}\right\rfloor.
\end{align*}
Summing these two inequalities directly gives \eqref{eq:generalized-antigriesmer}.
\end{proof}

The induction also determines the equality cases. For $r=1$, see \cite[Theorem~3.3(2) and Proposition~3.4]{DengHuangXiang2026}. The formulation below generalizes the corresponding statements for every $r$ and for maximum support weights.

\begin{theorem}\label{thm:projective-point-equality}
Let $1\le r<k$.
\begin{enumerate}
\item Suppose that $C$ is an $r$-th generalized Griesmer code, and let $P$ be a point of maximum multiplicity in $\mathcal{M}_C$. Then
\[
\mathcal{M}_C(P)=\left\lceil\frac{d_r}{q^{k-r}{r\brack 1}_q}\right\rceil,\qquad n(C_P)=g_q^{(r)}(k-1,d_r),\qquad d_r(C_P)=d_r.
\]
Consequently, $C_P$ has dimension $k-1$ and is an $r$-th generalized Griesmer code.
\item Suppose that $C$ is an $r$-th generalized antiGriesmer code, and let $P$ be a point of minimum multiplicity in $\mathcal{M}_C$. Then
\[
\mathcal{M}_C(P)=\left\lfloor\frac{\delta_r}{q^{k-r}{r\brack 1}_q}\right\rfloor,\qquad n(C_P)=a_q^{(r)}(k-1,\delta_r),\qquad \delta_r(C_P)=\delta_r.
\]
Consequently, $C_P$ has dimension $k-1$ and is an $r$-th generalized antiGriesmer code.
\end{enumerate}
\end{theorem}

\begin{proof}
Suppose first that $C$ is an $r$-th generalized Griesmer code, and put $\gamma=\mathcal{M}_C(P)$. Applying the generalized Griesmer bound to $C_P$ and using Lemma~\ref{lem:point-averaging} gives
\begin{align*}
n
&=(n-\gamma)+\gamma\\
&\ge g_q^{(r)}\bigl(k-1,d_r(C_P)\bigr)
+\left\lceil\frac{d_r}{q^{k-r}{r\brack 1}_q}\right\rceil\\
&\ge g_q^{(r)}(k-1,d_r)
+\left\lceil\frac{d_r}{q^{k-r}{r\brack 1}_q}\right\rceil\\
&=g_q^{(r)}(k,d_r)
=n.
\end{align*}
Hence, equality holds throughout. In particular, $\mathcal{M}_C(P)=\left\lceil d_r/(q^{k-r}{r\brack 1}_q)\right\rceil$ and $n(C_P)=g_q^{(r)}(k-1,d_r)$. Since $g_q^{(r)}(k-1,x)$ is strictly increasing in $x$, we also have $d_r(C_P)=d_r$.

Suppose next that $C$ is an $r$-th generalized antiGriesmer code, and put $\mu=\mathcal{M}_C(P)$. Applying the generalized antiGriesmer bound to $C_P$ and using Lemma~\ref{lem:point-averaging} gives
\begin{align*}
n
&=(n-\mu)+\mu\\
&\le a_q^{(r)}\bigl(k-1,\delta_r(C_P)\bigr)
+\left\lfloor\frac{\delta_r}{q^{k-r}{r\brack 1}_q}\right\rfloor\\
&\le a_q^{(r)}(k-1,\delta_r)
+\left\lfloor\frac{\delta_r}{q^{k-r}{r\brack 1}_q}\right\rfloor\\
&=a_q^{(r)}(k,\delta_r)
=n.
\end{align*}
Again, every inequality is an equality. In particular, $\mathcal{M}_C(P)=\left\lfloor\delta_r/(q^{k-r}{r\brack 1}_q)\right\rfloor$ and $n(C_P)=a_q^{(r)}(k-1,\delta_r)$. Since $a_q^{(r)}(k-1,x)$ is strictly increasing in $x$, we also have $\delta_r(C_P)=\delta_r$.
\end{proof}

Equality at one level also determines the extremal point multiplicities from every later support weight.

\begin{corollary}\label{cor:point-multiplicity-higher}
Let $r\le h<k$.
\begin{enumerate}
\item If $C$ is an $r$-th generalized Griesmer code, then $\gamma(\mathcal{M}_C)=\left\lceil d_h/(q^{k-h}{h\brack 1}_q)\right\rceil$. Moreover, for every $h$-dimensional subcode $A=\{xG:x\in U\}$, the set $\mathcal{P}\setminus\mathrm{PG}(U^\perp)$ contains a point of maximum multiplicity.
\item If $C$ is an $r$-th generalized antiGriesmer code, then $\mu(\mathcal{M}_C)=\left\lfloor\delta_h/(q^{k-h}{h\brack 1}_q)\right\rfloor$. Moreover, for every $h$-dimensional subcode $A=\{xG:x\in U\}$, the set $\mathcal{P}\setminus\mathrm{PG}(U^\perp)$ contains a point of minimum multiplicity.
\end{enumerate}
\end{corollary}

\begin{proof}
Equality at level $r$ implies equality at level $h$ by Corollary~\ref{cor:equality-propagation}. Applying Theorem~\ref{thm:projective-point-equality} with $h$ in place of $r$ gives both formulas.

For an $h$-dimensional subcode $A=\{xG:x\in U\}$, Lemma~\ref{lem:point-averaging} gives a point in $\mathcal{P}\setminus\mathrm{PG}(U^\perp)$ whose multiplicity is at least
\[
\left\lceil\frac{w(A)}{q^{k-h}{h\brack 1}_q}\right\rceil
\ge\left\lceil\frac{d_h}{q^{k-h}{h\brack 1}_q}\right\rceil
=\gamma(\mathcal{M}_C)
\]
in the Griesmer case. Hence, this point has maximum multiplicity. In the antiGriesmer case, the same lemma gives a point in $\mathcal{P}\setminus\mathrm{PG}(U^\perp)$ whose multiplicity is at most
\[
\left\lfloor\frac{w(A)}{q^{k-h}{h\brack 1}_q}\right\rfloor
\le\left\lfloor\frac{\delta_h}{q^{k-h}{h\brack 1}_q}\right\rfloor
=\mu(\mathcal{M}_C),
\]
so this point has minimum multiplicity.
\end{proof}

\section*{Acknowledgments}

The author would like to thank Tao Feng and Sascha Kurz for their helpful suggestions. The author is partially supported by the National Key R\&D Program of China under Grant No.~2025YFA1017700, the National Natural Science Foundation of China under Grant No.~123B2011, and the Postdoctoral Fellowship Program and China Postdoctoral Science Foundation under Grant No.~BX20250059.

\begin{singlespace}

\end{singlespace}


\begin{thebibliography}{99}

\bibitem{AhlswedeKhachatrian1998}
R. Ahlswede and L. H. Khachatrian,
\emph{The diametric theorem in Hamming spaces---optimal anticodes},
Adv. in Appl. Math. \textbf{20} (1998), no.~4, 429--449.

\bibitem{ChenXie2025}
H. Chen and C. Xie,
\emph{Projective linear codes and their simplex complementary codes},
J. Algebra \textbf{673} (2025), 304--320.

\bibitem{Delsarte1973}
P. Delsarte,
\emph{An algebraic approach to the association schemes of coding theory},
Philips Res. Rep. Suppl. No.~10 (1973), 97 pp.

\bibitem{DengHuangXiang2026}
H. Deng, H. Huang, and Q. Xiang,
\emph{Divisibility of Griesmer codes},
J. Combin. Theory Ser. A \textbf{222} (2026), Paper No.~106181.

\bibitem{DodunekovSimonis1998}
S. M. Dodunekov and J. Simonis,
\emph{Codes and projective multisets},
Electron. J. Combin. \textbf{5} (1998), no.~1, Research Paper 37, 23 pp.

\bibitem{Griesmer1960}
J. H. Griesmer,
\emph{A bound for error-correcting codes},
IBM J. Res. Develop. \textbf{4} (1960), no.~5, 532--542.

\bibitem{HellesethKloveLevenshteinYtrehus1995}
T. Helleseth, T. Kl{\o}ve, V. I. Levenshtein, and \O. Ytrehus,
\emph{Bounds on the minimum support weights},
IEEE Trans. Inform. Theory \textbf{41} (1995), no.~2, 432--440.

\bibitem{HellesethKloveYtrehus1992}
T. Helleseth, T. Kl{\o}ve, and \O. Ytrehus,
\emph{Generalized Hamming weights of linear codes},
IEEE Trans. Inform. Theory \textbf{38} (1992), no.~3, 1133--1140.

\bibitem{Kleitman1966}
D. J. Kleitman,
\emph{On a combinatorial conjecture of Erd\H{o}s},
J. Combinatorial Theory \textbf{1} (1966), no.~2, 209--214.

\bibitem{KurzLandjevRousseva2026}
S. Kurz, I. Landjev, and A. Rousseva,
\emph{Optimal codes and arcs for the generalized Hamming weights},
arXiv:2601.00250 [math.CO], 2026.

\bibitem{Nogin1999}
D. Yu. Nogin,
\emph{Higher weights of anticodes and the generalized Griesmer bound},
Finite Fields Appl. \textbf{5} (1999), no.~4, 409--423.

\bibitem{SolomonStiffler1965}
G. Solomon and J. J. Stiffler,
\emph{Algebraically punctured cyclic codes},
Information and Control \textbf{8} (1965), no.~2, 170--179.

\bibitem{TsfasmanVladut1995}
M. A. Tsfasman and S. G. Vl\u{a}du\c{t},
\emph{Geometric approach to higher weights},
IEEE Trans. Inform. Theory \textbf{41} (1995), no.~6, 1564--1588.

\bibitem{Ward1981}
H. N. Ward,
\emph{Divisible codes},
Arch. Math. (Basel) \textbf{36} (1981), no.~6, 485--494.

\bibitem{Ward1998}
H. N. Ward,
\emph{Divisibility of codes meeting the Griesmer bound},
J. Combin. Theory Ser. A \textbf{83} (1998), no.~1, 79--93.

\bibitem{Wei1991}
V. K. Wei,
\emph{Generalized Hamming weights for linear codes},
IEEE Trans. Inform. Theory \textbf{37} (1991), no.~5, 1412--1418.

\bibitem{XieChenDingLi2026}
C. Xie, H. Chen, C. Ding, and C. Li,
\emph{AntiGriesmer bounds, optimal codes, and their subcode support weight distributions},
IEEE Trans. Inform. Theory \textbf{72} (2026), no.~4, 2133--2143.

\bibitem{ZhangChenLinLiu2026}
G. Zhang, B. Chen, L. Lin, and H. Liu,
\emph{Improved antiGriesmer bounds for linear anticodes and applications},
J. Algebra \textbf{707} (2026), 145--164.

\end{thebibliography}
\end{document}